\documentclass[reqno]{amsart}

\usepackage[utf8]{inputenc}
\usepackage[T1]{fontenc}
\usepackage{amsthm}
\usepackage{amsmath}
\usepackage{amssymb}
\usepackage{enumitem}
\usepackage[dvips]{graphicx}
\usepackage{color}
\usepackage{hyperref}
\usepackage{url}
\usepackage{stackrel}
\usepackage[left=3.5cm, right=3.5cm, paperheight=11.5in]{geometry}
\usepackage{fancyhdr}
\usepackage{mathrsfs}
\usepackage{stmaryrd}
\usepackage{soul}
\usepackage{nicefrac}
\usepackage{datetime}
\longdate

\newtheorem{theorem}{Theorem}
\newtheorem*{main*}{Main Theorem}
\newtheorem{lemma}{Lemma}
\newtheorem{corollary}{Corollary}

\theoremstyle{definition}

\newtheorem*{def*}{Definition}
\newtheorem{example}{\textsc{Example}}

\newtheorem*{question*}{\textbf{Question}}

\theoremstyle{remark}

\newtheorem*{claim*}{\textsc{Claim}}

\hypersetup{
    pdftitle={Upper and lower densities have the strong Darboux property},
    pdfauthor={Paolo Leonetti and Salvatore Tringali},
    pdfmenubar=false,
    pdffitwindow=true,
    pdfstartview=FitH,
    colorlinks=true,
    linkcolor=blue,
    citecolor=green,
    urlcolor=cyan
}

\DeclareMathSymbol{\widehatsym}{\mathord}{largesymbols}{"62}

\renewcommand{\bf}{\mathbf}
\renewcommand{\emptyset}{\varnothing}

\renewcommand{\rho}{\varrho}
\renewcommand{\ast}{\star}

\providecommand{\HHb}{\mathbf{H}}

\providecommand{\dd}{\mathsf{d}}

\providecommand{\DDc}{\mathcal{D}}

\providecommand{\HHc}{\mathcal{H}}

\providecommand{\LLc}{\mathscr{L}}

\providecommand{\NNb}{\mathbf{N}}

\providecommand{\PPc}{\mathcal{P}}

\providecommand{\RRb}{\mathbf{R}}

\providecommand{\VVc}{\mathcal{V}}

\providecommand{\ZZb}{\mathbf{Z}}

\renewcommand{\xi}{a}

\providecommand\dom{{\rm dom}}
\providecommand\llb{\llbracket}
\providecommand\rrb{\rrbracket}

\providecommand{\imag}{{\rm Im}}

\newcommand\pto{\mathrel{\ooalign{\hfil$\mapstochar$\hfil\hfil\cr$\to$\cr}}}

\newcommand{\fixed}[2][1]{%
  \begingroup
  \spaceskip=#1\fontdimen2\font minus \fontdimen4\font
  \xspaceskip=0pt\relax
  #2%
  \endgroup
}

\begin{document}
\title{Upper and lower densities have the \\ strong Darboux property}

\author{Paolo Leonetti}
\address{Universit\`a ``Luigi Bocconi'' \\ via Sarfatti 25, 20136 Milano, Italy}
\email{leonetti.paolo@gmail.com}

\author{Salvatore Tringali}
\address{Department of Mathematics, Texas A\&M University at Qatar \\ PO Box 23874 Doha, Qatar}
\curraddr{Institute for Mathematics and Scientific Computing, University of Graz | Heinrichstr. 36, 8010 Graz, Austria}
\email{salvatore.tringali@uni-graz.at}
\urladdr{http://imsc.uni-graz.at/tringali}

%
%
\begin{abstract}
\noindent{} Let $\mathcal{P}({\bf N})$ be the power set of $\bf N$. An upper density (on $\bf N$) is a non\-decreasing and subadditive function $\mu^\ast: \mathcal{P}({\bf N})\to\bf R$ such that $\mu^\ast({\bf N}) = 1$ and $\mu^\ast(k \cdot X + h) = \frac{1}{k} \mu^\ast(X)$ for all $X \subseteq \bf N$ and $h,k \in {\bf N}^+$, where $k \cdot X + h := \{kx + h: x \in X\}$.

The upper asymptotic, upper Banach, upper logarithmic, upper Buck, upper P\'olya, and upper analytic densities are examples of upper densities.

We show that every upper density $\mu^\ast$ has the strong Darboux property, and so does the associated lower density, where a function $f: \mathcal P({\bf N}) \to \bf R$ is said to have the strong Darboux property if, whenever $X \subseteq Y \subseteq \bf N$ and $a \in [f(X),f(Y)]$, there is a set $A$ such that $X\subseteq A\subseteq Y$ and $f(A)=a$. In fact, we prove the above under the assumption that the monotonicity of $\mu^\ast$ is relaxed to the weaker condition that $\mu^\ast(X) \le 1$ for every $X \subseteq \bf N$.
\end{abstract}
\maketitle
\thispagestyle{empty}

\section{Introduction}\label{sec:introduction}
Recently, the authors have introduced, and studied some fundamental aspects of, an ``axiomatic theory of densities'' \cite{LT},
tailored to the integers
and built around the properties of certain (set) functions called ``upper [quasi-]densities''.
These functions are also the subject of the present manuscript, as our main goal is to prove that they satisfy a kind of ``intermediate value property'' we refer to as the strong Darboux property, see \S{} \ref{sec:darboux} below.

Throughout, we will let $\HHb$ be either $\ZZb$, $\NNb$, or $\NNb^+$. A function $\mu^\ast: \mathcal{P}(\HHb) \to \RRb$ is said to be an \emph{upper density} (on $\HHb$) if, for all $X,Y \subseteq \HHb$ and $h,k \in \NNb^+$, it holds:
\begin{enumerate}[label={\rm (\textsc{f}\arabic{*})}]
\item\label{item:F1} $\mu^\ast(\HHb) = 1$;
\item\label{item:F2} $\mu^\star(X) \le \mu^\star(Y)$ for $X\subseteq Y$;
\item\label{item:F3} $\mu^\ast(X \cup Y) \le \mu^\ast(X) + \mu^\ast(Y)$;
\item\label{item:F4} $\mu^\ast(k\cdot X) = \frac{1}{k}\mu^\ast(X)$, where $k\cdot X:=\{kx\colon x \in X\}$;
\item\label{item:F5} $\mu^\ast(X+h) = \mu^\ast(X)$, where $X + h := \{x+h: x \in X\}$.
\end{enumerate}
We regard $\HHb$ as a sort of ``parameter'' allowing for different scenarios and some flexibility, though it makes almost no difference for this paper to focus on the case $\HHb = \NNb$.

For future reference, notice that \ref{item:F4} and \ref{item:F5} together are equivalent to the following:
\begin{enumerate}[label={\rm (\textsc{f}4${}^\flat$)}]
\item\label{item:F4b} $\mu^\ast(k \cdot X+h) = \frac{1}{k}\mu^\ast(X)$ for all $X \subseteq \HHb$ and $h, k \in \NNb^+$.
\end{enumerate}
To ease the exposition, we will occasionally say that $\mu^\ast$ is:
monotone if it satisfies \ref{item:F2};
subadditive if it satisfies \ref{item:F3}; and
(finitely) additive if $\mu^\ast(X \cup Y) = \mu^\ast(X) + \mu^\ast(Y)$ whenever $X, Y \subseteq \HHb$ and $X \cap Y$ is empty.
On the other hand, we call $\mu^\ast$ an \textit{upper quasi-density} (on $\HHb$) if it satisfies \ref{item:F1}, \ref{item:F3}, \ref{item:F4b}, and the next condition, which is clearly implied by \ref{item:F1} and \ref{item:F2}:
\begin{enumerate}[label={\rm (\textsc{f}2${}^\flat$)}]
\item\label{item:F2b} $\mu^\star(X) \le 1$ for every $X \subseteq \HHb$.
\end{enumerate}
Note that if $\mu^\ast$ is an upper quasi-density, then $\mu^\ast(X) = \mu^\ast(X \cup X) \le 2\mu^\ast(X)$ by \ref{item:F3}, which together with \ref{item:F2b} implies that the image of $\mu^\ast$ is contained in the interval $[0,1]$.

Of course, every upper density is an upper quasi-density, and it can be proved that non-monotone upper quasi-densities do actually exist, see \cite[Theorem 1]{LT}.
In particular, the upper asymptotic (or natural), upper Banach (or uniform), upper logarithmic, upper Buck, upper P\'olya, and upper analytic densities, together with the upper $\alpha$-densities ($\alpha \ge -1$), are upper densities in the sense of the above definitions, see \cite[\S{} 3 and Examples 4-6 and 8]{LT}.

All these densities, along with many others, have been the subject of a great deal of research, and some of them have played a prominent role in the development of (probabilistic and additive) number theory and some areas of analysis and ergodic theory. Roughly speaking, one reason for this is that densities provide an effective alternative to measures in the study of the interrelation between the ``structure'' of a set of integers and some kind of information about its ``largeness'', a general principle that is brightly illustrated by the Erd\H{o}s-Tur\'an conjecture \cite[\S{} 35.4]{Soif} that any set $X$ of positive integers such that $\sum_{x \in X} \frac{1}{x} = \infty$ contains arbitrarily long arithmetic progressions; see also \cite{LT} and references therein for other pointers.

Given an upper [quasi-]density $\mu^\ast$ on $\HHb$, we refer to $\mu_\ast: \PPc(\HHb) \to \RRb: X \mapsto 1 - \mu^\ast(X^c)$ as the \emph{lower \textup{[}quasi-\textup{]}density} associated to $\mu^\star$. It has been shown in \cite[Theorem 2]{LT} that the image of $\mu^\ast$, and hence also of $\mu_\ast$, is the entire interval $[0,1]$. The ultimate goal of the present paper is to prove that these results can be considerably strengthened.

Before proceeding to details, a remark is in order: While it is arguable that non-monotone quasi-densities are not very interesting from the point of view of applications, it seems meaningful to establish if certain properties of a specific class of objects depend or not on a particular assumption (in the present case of interest, the axiom of monotonicity), as this may provide a better understanding of the objects themselves.
That is basically our motivation for considering upper quasi-densities, instead of restricting our attention to upper densities.

We refer to \cite{LT}, and particularly to \S\S{} 2--3 therein, for notation, terminology and conventions used, but not explained, in this work. Note that measures will always be functions $\Sigma \to \RRb$ for which $\Sigma$ is a sigma-algebra, but unless otherwise specified, they do not need to be nonnegative or countably additive (that is, they may be signed or finitely additive). 

We use the letters $h$, $i$, and $k$, with or without subscripts, for non\-negative integers, the letter $n$ for a positive integer, and the letter $s$ for a positive real. In our notations, $0 \in \NNb$.
\section{Darboux properties}
\label{sec:darboux}
We start this section with a couple more of definitions.
Given a set $S$, we say that a partial function $f: \PPc(S) \pto \RRb$ with domain $\DDc$ has:
\begin{enumerate}[label={\rm (\textsc{d}\arabic{*})}]
\item\label{item:weak_Darboux} the \textit{weak Darboux property} if $\emptyset \in \DDc$ and for every $X \in \DDc$ and $a \in [f(\emptyset),f(X)]$ there is a set $A \in \DDc$ such that $A \subseteq X$ and $f(A) = a$;
\item\label{item:Darboux} the \textit{strong Darboux property} if for all $X, Y \in \DDc$ with $X \subseteq Y$ and every $a \in [f(X),f(Y)]$ there exists a set $A \in \DDc$ such that $X\subseteq A\subseteq Y$ and $f(A) = a$.
\end{enumerate}
Of course, \ref{item:weak_Darboux} is implied by \ref{item:Darboux} provided that $\emptyset \in \DDc$, and the converse is true, e.g., for finitely additive measures. Notice also that, since $f$ does not need to be monotone in condition \ref{item:weak_Darboux}, it may well happen that $f(X) < f(\emptyset)$ for some $X \in \DDc$, in which case $[f(\emptyset), f(X)]$ is empty and there is nothing to prove; analogous considerations apply to the strong Darboux property.

Some authors, either in measure theory, see, e.g., \cite[\S{} V.46.I, Corollary 3${}^\prime$]{Kura} and \cite[\S{} I.2.9, Definition 4]{Dinc}, or in connection to the study of densities in number theory, see, for instance, \cite[\S{} 2]{PaSa}, \cite[p. 217]{MMT}, and \cite{GMT}, refer to \ref{item:weak_Darboux} as simply the Darboux
property (notice that \cite{PaSa} points to \cite{Dinc}, \cite{MMT} points to \cite{Kura}, and \cite{GMT} points to \cite{MMT} as a source for the terminology). But that does not sound very fit to us, as \ref{item:Darboux} is arguably closer than \ref{item:weak_Darboux} to the spirit of the intermediate value property of real-valued functions of a real variable, so we prefer sticking to our own definitions.

Other terms of common usage to allude to condition \ref{item:weak_Darboux}
are ``full-valued'', see, e.g., \cite[p. 174]{M53}, 
and ``strongly non-atomic'', see, e.g., \cite[Definition 5.1.5]{RaRa}.

We note that if $f$ has the weak Darboux property, then $f(X) \le f(\emptyset)$ for every finite $X \subseteq S$: If $f(\emptyset) < f(X)$ for some $X \subseteq S$, then the interval $[f(\emptyset), f(X)]$ has positive width, so $X$ should have infinitely many subsets for $f$ to have the weak Darboux property, which, however, is not the case when $X$ is finite. In addition, we have the following elementary result:
\begin{lemma}
\label{lm:Darboux_&_duality}
Let $\mu^\ast$ be a function $\PPc(\HHb) \to \RRb$ and $\mu_\ast$ its conjugate $\PPc(\HHb) \to \RRb: X\mapsto 1-\mu^\ast(X^c)$. Then $\mu^\ast$ has the strong \textup{(}respectively, weak\textup{)} Darboux property if and only if $\mu_\ast$ does.
\end{lemma}
\begin{proof}
Since $\mu^\ast$ is the conjugate of $\mu_\ast$, it is enough to assume, as we do, that $\mu^\ast$ has the strong (respectively, weak) Darboux property and to prove that the same is true for $\mu_\ast$.

For this, fix $X, Y \subseteq \HHb$ with $X \subseteq Y$ and $a \in [\mu_\ast(X), \mu_\ast(Y)]$, and set $X := \emptyset$ if $\mu^\ast$ has just the weak Darboux property. Then $Y^c \subseteq X^c$, and it follows from the hypotheses, and from the fact that $\mu_\ast(S) + \mu^\ast(S^c) = 1$ for every $S \subseteq \HHb$, that there exists $A \subseteq \HHb$ such that $Y^c \subseteq A \subseteq X^c$ and $1 - \mu_\ast(A^c) = \mu^\ast(A) = 1-a$.
Hence $\mu_\ast(A^c) = a$, which ultimately yields that $\mu_\ast$ has the strong (respectively, weak) Darboux property, when considering that $X \subseteq A^c \subseteq Y$.
\end{proof}
With this in mind, here comes the main contribution of the present paper, which is reminiscent of results of D. Maharam \cite[Theorem 2]{M76}, on finitely additive nonnegative measures\footnote{: In Maharam's own words, the result applies to ``arbitrary (finite, finitely additive) measures on an arbitrary field of subsets of an arbitrary set'', see \cite[p. 49]{M76}. This may sound confusing, as the proof of the theorem uses that the domain of a measure is closed, in particular, under countable unions, and on the other hand, the term ``field of sets'' is now commonly meant for a pair $(S,\mathscr{F})$, where $S$ is a set and $\mathscr{F}$ a subfamily of $\PPc(S)$ that contains $S$ and is closed under finite unions, finite intersections, and absolute complements. The fact is just that Maharam uses the term to refer to what we now call a sigma-algebra, as can be argued by previous work of her \cite{M42}.}, and G. Sikorski
\cite[Problem 12, p. 225]{Sik}, on non-atomic, countably additive, nonnegative measures\footnote{: The latter result is frequently attributed to W. Sierpi\'nski, and in so doing a reference to \cite{Sier} is often provided. But it is unclear why this should be correct, as Sierpi\'nski's paper deals with the convexity of the range of certain \textit{finitely} additive measures $\mathcal P(S) \to \RRb$ for which $S$ is a bounded subset of $\RRb^n$. The (alleged) mistake may have originated from an incorrect interpretation of a footnote on the very same page of Sikorski's book \cite{Sik}, where the reader is addressed to \cite{Sier} and \cite{Fich}, though most likely for the sake of comparison.}.
%
\begin{theorem}
\label{th:upperdarboux}
Every upper quasi-density has the strong Darboux property.
\end{theorem}
%
The theorem, which is proved in \S{} \ref{sec:proofdarboux}, is a major generalization of \cite[Theorem 2]{LT}, and it leads, together with Lemma \ref{lm:Darboux_&_duality}, to the following results (we omit further details):
\begin{corollary}
\label{cor:lowerdarboux}
Every lower quasi-density has the strong Darboux property.
\end{corollary}
In particular, the next corollary is now obvious,
but we record it here for future reference:
\begin{corollary}
\label{cor:weak_Darboux}
Upper and lower quasi-densities have the weak Darboux property.
\end{corollary}
Special instances of Corollary \ref{cor:weak_Darboux} have already appeared in the literature, the proofs of these former results being based on ad hoc arguments tailored to the particular densities under consideration (contrary to the proof of Theorem \ref{th:upperdarboux}). To be more precise, it is known from work of G.~Grekos, see \cite{GrekThes, Grek78}, that the upper $\alpha$-densities (on $\NNb^+$) have the weak Darboux property for every real $\alpha \ge -1$, and this has been later extended (and in a stronger form) to certain weighted densities (on $\NNb^+$), see \cite[Proposition 1]{GMT}.
%
%
\section{Proof of Theorem \ref{th:upperdarboux}}
\label{sec:proofdarboux}
We will need the following results from \cite[\S{} 6]{LT}, which yield that upper quasi-densities, though not necessarily monotone, satisfy a kind of ``weak monotonicity'' (we include some of the proofs for completeness).

For ease of notation, we denote by $\VVc_{k,\HHc}$, for all $k \in \NNb^+$ and $\HHc \subseteq \NNb$, the set $\bigcup_{h\fixed[0.1]{\text{ }}\in\fixed[0.1]{\text{ }}\HHc} (k \cdot \HHb + h)$, and write $\VVc_k$ in place of $\VVc_{k,\HHc}$ when $\HHc = \llb 0, k-1 \rrb$.
\begin{lemma}
\label{lem:mu_of_finite_sets}
If $\mu^\ast$ is an upper quasi-density, then $\mu^\ast(X) = 0$ for every finite $X \subseteq \HHb$.
\end{lemma}
\begin{proof}
If $X = \emptyset$, or $0 \in \HHb$ and $X = \{0\}$, then $k \cdot X = X$ for all $k \ge 1$, and we have from \ref{item:F4} that $2\mu^\ast(X) = \mu^\ast(X)$, i.e. $\mu^\ast(X) = 0$. On the other hand, we get by \ref{item:F4b} that, for every $k \ge 1$,
$$
\mu^\ast(\{1\}) = \mu^\ast(\{1\} + k-1) = \mu^\ast(\{k\}) = \mu^\ast(k \cdot \{1\}) = \frac{1}{k} \mu^\ast(\{1\}),
$$
which is possible only if $\mu^\ast(\{1\}) = 0$. Consequently, $\mu^\ast(\{k\}) = 0$ for all $k \ge 0$, and in addition, if $\HHb = \ZZb$ and $k \le 0$, then $\mu^\ast(\{k\}) = \mu^\ast(\{k\} + (-k)) = \mu^\ast(\{0\}) = 0$.

Since upper quasi-densities are non-negative functions (as was noted in the introduction), it follows that if $X \subseteq \HHb$ is finite then $
0 \le \mu^\ast(X) \le  \sum_{x \fixed[0.2]{\text{ }}\in X} \mu^\ast(\{x\}) = 0$, hence $\mu^\ast(X) = 0$.
\end{proof}
\begin{lemma}
\label{lem:unions_of_translates}
Let $\mu^\ast$ be an upper quasi-density, and for a fixed $k \in \NNb^+$ let $h_1, \ldots, h_n \in \NNb$ be such that $h_i \not\equiv h_j \bmod k$ for $1 \le i < j \le n$. Then, $\mu^\ast(\mathcal{V}_{k,\mathcal H} \cup \mathcal{V}) = \frac{n}{k}$ and $\mu^\ast(\mathcal{V}_{k,\mathcal H} \setminus \mathcal{V}) \ge \frac{n}{k}$ for every finite $\mathcal{V} \subseteq \HHb$, where $\mathcal H := \{h_1, \ldots, h_n\}$.
\end{lemma}
\begin{proof}
Let $l_i$ be, for each $i=1, \ldots, n$, the remainder of the integer division of $h_i$ by $k$ (in such a way that $0 \le l_i < k$), and set $\mathcal{H}^c := \llb 0, k-1 \rrb \setminus \{l_1, \ldots, l_n\}$.

Clearly, $\HHb = \mathcal{V}_{k,\mathcal H} \cup \mathcal{V}_{k,\mathcal H^c} \cup S$ for some finite $S \subseteq \HHb$. Thus,
we get by \ref{item:F1}, \ref{item:F3}, and \ref{item:F4b} and Lemma \ref{lem:mu_of_finite_sets} that, however we choose a finite $\mathcal{V} \subseteq \HHb$,
\begin{equation*}
\begin{split}
1 = \mu^\ast(\HHb)
    & \le \mu^\ast(\mathcal{V}_{k,\mathcal H} \cup \mathcal{V}) + \mu^\ast(\mathcal{V}_{k,\mathcal H^c} \cup S) \le \mu^\ast(\mathcal{V}_{k,\mathcal H}) + \mu^\ast(\mathcal{V}_{k,\mathcal H^c}) + \mu^\ast(S) + \mu^\ast(\mathcal{V}) \\
    & = \mu^\ast(\mathcal{V}_{k,\mathcal H}) + \mu^\ast(\mathcal{V}_{k,\mathcal H^c}) \le n \mu^\ast(k \cdot \HHb) + (k-n) \mu^\ast(k \cdot \HHb) = k \mu^\ast(k \cdot \HHb) = 1,
\end{split}
\end{equation*}
which is possible only if $\mu^\ast(\mathcal{V}_{k,\mathcal H} \cup \mathcal{V}) = \mu^\ast(\mathcal{V}_{k,\mathcal H}) = n\mu^\ast(k \cdot \HHb) = \frac{n}{k}$.

As for the rest, we now have from \ref{item:F5}, Lemma \ref{lem:mu_of_finite_sets}, and the above that, for every finite $\mathcal V \subseteq \HHb$, $\frac{n}{k} = \mu^\ast(\mathcal{V}_{k,\mathcal H}) \le \mu^\ast(\mathcal{V}_{k,\mathcal H} \setminus \mathcal{V}) + \mu^\ast(\mathcal V) = \mu^\ast(\mathcal{V}_{k,\mathcal H} \setminus \mathcal V)$, which completes the proof.
\end{proof}
\begin{lemma}
\label{lm:relaxed_(F2)_with_upper_arithmetic_bound}
Let $\mu^\ast$ be an upper quasi-density on $\HHb$, and pick $X, Y \subseteq \HHb$ such that $X \subseteq Y$ and $Y$ is a finite union of arithmetic progressions of $\HHb$, or differs from a set of this form by finitely many integers. Then $\mu^\ast(X) \le \mu^\ast(Y)$.
\end{lemma}
\begin{proof}
By hypothesis, there exist $k \in \NNb^+$ and $\HHc \subseteq \llb 0, k-1 \rrb$ such that the symmetric difference of $Y$ and $\mathcal V_{k,\mathcal H}$ is finite.
Using that $X \subseteq Y$, this yields that the relative complement of $\bigcup_{h \fixed[0.2]{\text{ }}\in\fixed[0.1]{\text{ }} \HHc} X_h$ in $X$, where $X_h := X \cap (k \cdot \HHb + h) \subseteq X$ for each $h \in \HHc$, is finite too. Therefore, we obtain that $\mu^\ast(X) \le \sum_{h \fixed[0.2]{\text{ }}\in\fixed[0.1]{\text{ }} \HHc}\mu^\ast(X_h)$ by Lemma \ref{lem:mu_of_finite_sets} and \ref{item:F3}, and $\mu^\ast(Y) \ge \frac{1}{k}|\HHc|$ by Lemma \ref{lem:unions_of_translates}.

On the other hand, we have that, however we choose $h \in \HHc$, there is a set $S_h \subseteq \HHb$ for which $X_h = k \cdot S_h + h$, so we get from the above, \ref{item:F4b} and the fact that $\imag(\mu^\ast) \subseteq [0,1]$ that
$$
\mu^\ast(X) \le \sum_{h \fixed[0.2]{\text{ }}\in\fixed[0.1]{\text{ }} \HHc} \mu^\ast(X_h) = \sum_{h \fixed[0.2]{\text{ }}\in\fixed[0.1]{\text{ }} \HHc} \mu^\ast(k \cdot S_h + h) = \frac{1}{k} \sum_{h \fixed[0.2]{\text{ }}\in\fixed[0.1]{\text{ }} \HHc} \mu^\ast(S_h) \le \mu^\ast(Y),
$$
which completes the proof.
\end{proof}
We use Lemma \ref{lm:relaxed_(F2)_with_upper_arithmetic_bound} in the proof of the next result, which may be of independent interest, insofar as it can be adapted to prove other statements along the lines of Theorem \ref{th:upperdarboux}, but for different classes of ``densities'' than the ones picked up in this work.
\begin{lemma}
\label{lem:increasing}
Let $\mu^\ast$ be an upper quasi-density on $\HHb$, pick $A, B \subseteq \HHb$ with $A \subseteq B$ and $\mu^\ast(A) < \mu^\ast(B)$,
and fix $a, b \in \RRb$ such that $\mu^\ast(A) \le a < b \le \mu^\ast(B)$. Then, for every $k > (b-a)^{-1}$ there exist $\HHc_0 \subseteq \llb 0, k-1 \rrb$ and $h_0 \in \llb 0, k-1 \rrb$ with the property that
$$
a < \mu^\ast(A \cup (B \cap \VVc_{k,\HHc_0})) < b \le \mu^\ast(A \cup (B \cap \VVc_{k,\HHc_0 \fixed[0.2]{\text{ }} \cup \fixed[0.2]{\text{ }} \{h_0\}})).
$$
\end{lemma}
\begin{proof}
Fix an integer $k > (b-a)^{-1}$ and let $\LLc_k$ be the set of all subsets $\HHc$ of $\llb 0, k-1 \rrb$ for which $\mu^\ast(A \cup (B \cap \VVc_{k,\HHc})) > a$.
We have that $\HHb = S \cup \VVc_k$ for some finite set $S$; in fact, $S = \llb 1, k-1 \rrb$ if $\HHb = \NNb^+$, and $S = \emptyset$ otherwise. So it follows from \ref{item:F3} and Lemma \ref{lm:relaxed_(F2)_with_upper_arithmetic_bound} that
\begin{equation}
\label{equ:majorization}
\begin{split}
a < \mu^\ast(B)
    & = \mu^\ast(A \cup B) = \mu^\ast(A \cup (B \cap (S \cup \VVc_k))) = \mu^\ast(A \cup (B \cap \VVc_k) \cup (B \cap S)) \\
    & \le \mu^\ast(A \cup (B \cap \VVc_k)) + \mu^\ast(B \cap S) = \mu^\ast(A \cup (B \cap \VVc_k)).
\end{split}
\end{equation}
This implies that $\llb 0, k-1 \rrb \in \LLc_k$. Let $\HHc_a$ be a set of minimal cardinality in $\LLc_k$ (which exists since $\LLc_k$ is finite and nonempty), and observe that $\HHc_a$ is nonempty, otherwise we would have that $\VVc_{k,\HHc_a} = \emptyset$, and hence $a < \mu^\ast(A \cup (B \cap \VVc_{k,\HHc_a})) = \mu^\ast(A) \le a$, which is absurd.

Now, suppose for the sake of a contradiction that $\mu^\ast(A \cup (B \cap \VVc_{k,\HHc_a})) \ge b$, and using that $\HHc_a \ne \emptyset$, fix $h \in \HHc_a$. Since $|\HHc_a \setminus \{h\}| < |\HHc_a|$ and $\HHc_a$ is, by construction, an element of minimal cardinality in $\LLc_k$, we must have that $\mu^\ast(A \cup (B \cap \VVc_{k,\HHc_a \setminus \{h\}})) \le a$.
But it holds that
\begin{equation*}
\begin{split}
A \cup (B \cap \VVc_{k,\HHc_a})
    & = A \cup (B \cap \VVc_{k, \HHc_a \setminus \{h\}}) \cup (B \cap (k \cdot \HHb + h)),
\end{split}
\end{equation*}
and so we get from \ref{item:F3}, \ref{item:F4b}, and Lemma \ref{lm:relaxed_(F2)_with_upper_arithmetic_bound} that
\begin{equation*}
\begin{split}
b \le \mu^\ast(A \cup (B \cap \VVc_{k,\HHc_a}))
    & \le \mu^\ast(A \cup (B \cap \VVc_{k,\HHc_a \setminus \{h\}})) + \mu^\ast(B \cap (k \cdot \HHb + h)) \\
    & \le a + \mu^\ast(k \cdot \HHb + h) = a + \frac{1}{k} < a + b-a = b,
\end{split}
\end{equation*}
which is a contradiction. It follows that $a < \mu^\ast(A \cup (B \cap \VVc_{k,\HHc_a})) < b$.

At this point, denote by $\mathscr{M}_k$ the set of all subsets $\HHc$ of $\llb 0, k-1 \rrb$ containing $\HHc_a$ and such that $a < \mu^\ast(A \cup (B \cap \VVc_{k, \HHc})) < b$; we know from the above that $\HHc_a \in \mathscr{M}_k$. Then, let $\HHc_0$ be a set of maximal cardinality in $\mathscr{M}_k$ (which exists since $\mathscr{M}_k$, too, is finite and nonempty), and $\mathscr{S}_k$ the set of all subsets $\HHc$ of $\llb 0, k-1 \rrb$ containing $\HHc_0$ and such that $\mu^\ast(A \cup (B \cap \VVc_{k,\HHc})) \ge b$; we have from \eqref{equ:majorization} that $\llb 0, k-1 \rrb \in \mathscr{S}_k$, as $\mu^\ast(A \cup (B \cap \VVc_k)) \ge \mu^\ast(B) \ge b$.
Accordingly, let $\HHc_b$ be an element of minimal cardinality in $\mathscr{S}_k$  (which exists, once again, since $\mathscr{S}_k$ is finite and nonempty).

It is clear that $|\HHc_b| \ge 1 + |\HHc_0|$, because $\HHc_0 \subseteq \HHc_b$ and $\mu^\ast(\HHc_0) < b \le \mu^\ast(\HHc_b)$, and consequently $\HHc_0 \subsetneq \HHc_b$; we claim that $|\HHc_b| = |\HHc_0| + 1$. In fact, suppose to the contrary that $|\HHc_b| \ge 2 + |\HHc_0|$, and pick $h \in \HHc_b \setminus \HHc_0$. Then $\HHc_0 \subsetneq \HHc_b \setminus \{h\} \subsetneq \HHc_b$, so $\HHc_0$ being an element of maximal cardinality in $\mathscr{M}_k$ and $\HHc_b$ an element of minimal cardinality in $\mathscr{S}_k$ entail that $\HHc_b \setminus \{h\} \notin \mathscr{M}_k \cup \mathscr{S}_k$.

Thus $\mu^\ast(A \cup (B \cap \VVc_{k, \HHc_b \setminus \{h\}})) \le a$, which is however impossible by the same argument used above to prove that $\mu^\ast(A \cup (B \cap \VVc_{k, \HHc_a})) < b$, as it would imply (we omit some details) that
\begin{equation*}
b \le \mu^\ast(A \cup (B \cap \VVc_{k, \HHc_b})) \le \mu^\ast(A \cup (B \cap \VVc_{k, \HHc_b \setminus \{h\}})) + \mu^\ast(B \cap (k \cdot \HHb + h)) < b.
\end{equation*}
This completes the proof of the lemma, as it shows that, letting $h_0$ be the unique element in $\HHc_b \setminus \HHc_0$, we have $a < \mu^\ast(A \cup (B \cap \VVc_{k, \HHc_0})) < b \le \mu^\ast(A \cup (B \cap \VVc_{k,\HHc_0 \fixed[0.2]{\text{ }} \cup \fixed[0.2]{\text{ }} \{h_0\}}))$.
\end{proof}
We are now ready to prove the main result of the paper.
\begin{proof}[Proof of Theorem \ref{th:upperdarboux}]
Let $\mu^\ast$ be an upper quasi-density on $\HHb$, and fix $X, Y \in \PPc(\HHb)$, $X \subseteq Y$. If $\mu^\ast(Y) \le \mu^\ast(X)$, the conclusion is trivial, so assume in what follows that $\mu^\ast(X) < \mu^\ast(Y)$ and let $\xi \in {]\mu^\ast(X), \mu^\ast(Y)[}$ (the boundary cases are obvious).
\begin{claim*} There exist two sequences $(A_n)_{n \ge 1}$ and $(B_n)_{n \ge 1}$ of subsets of $\HHb$ for which:
\begin{enumerate}[label={\rm (\roman{*})}]
\item\label{item:recursive_(i)} $X \subseteq A_1 \subseteq \cdots \subseteq A_n \subseteq \cdots \subseteq B_n \subseteq \cdots \subseteq B_1 \subseteq Y$;
\item\label{item:recursive_(ii)} $\mu^\ast(A_n) < \xi \le \mu^\ast(B_n)$ for all $n \in \NNb^+$;
\item\label{item:recursive_(iii)} Given $n \in \NNb^+$, there exist $h, k \in \NNb$ such that $k \ge n$ and $B_n \setminus A_n \subseteq k \cdot \HHb + h$.
\end{enumerate}
\end{claim*}
\begin{proof}[Proof of the claim]
To begin, set $A_1:=X$ and $B_1:=Y$, by noting that $\mu^\star(X) < \xi\le \mu^\star(Y) \le 1 \le 1 + \mu^\ast(X)$ and $Y \setminus X \subseteq 1 \cdot \HHb + 0$. Next, fix $v \in \NNb^+$ and suppose for the sake of induction that we have already found subsets $A_1,\ldots,A_v$ and $B_1,\ldots,B_v$ of $\HHb$ such that $A_1 \subseteq \cdots \subseteq A_v \subseteq B_v \subseteq \cdots \subseteq B_1$ and conditions \ref{item:recursive_(ii)} and \ref{item:recursive_(iii)} hold true for $1 \le n \le v$.

Let $k > (\xi-\mu^\star(A_v))^{-1}$. Our assumptions give that $\mu^\ast(A_v) < \mu^\ast(B_v)$ and $B_v \setminus A_v \subseteq q \cdot \HHb + r$ for some $q \in \NNb^+$ and $r \in \NNb$ such that $q \ge v$. Accordingly, we get from axioms \ref{item:F3}-\ref{item:F4b} and Lemma \ref{lm:relaxed_(F2)_with_upper_arithmetic_bound} that
$$
\mu^\ast(B_v) \le \mu^\ast(A_v) + \mu^\ast(B_v \setminus A_v) \le \mu^\ast(A_v) + \mu^\ast(q \cdot \HHb + r) = \mu^\ast(A_v) + \frac{1}{q} \le \mu^\ast(A_v) + \frac{1}{v},
$$
which in turn implies that $k > (\mu^\ast(B_v) - \mu^\ast(A_v))^{-1} \ge v$, i.e. $k \ge v+1$.
On the other hand, it follows from Lemma \ref{lem:increasing} that there exist $\HHc_0 \subseteq \llb 0, k-1 \rrb$ and $h_0 \in \NNb$ such that
\begin{equation}
\label{equ:inductive_step_disug}
\mu^\ast(A_v \cup (B_v \cap \VVc_{k,\HHc_0})) < \xi \le \mu^\ast(A_v \cup (B_v \cap \VVc_{k,\HHc_0 \fixed[0.2]{\text{ }} \cup \fixed[0.2]{\text{ }} \{h_0\}})).
\end{equation}
Therefore, set $A_{v+1} := A_v \cup (B_v \cap \VVc_{k,\HHc_0})$ and $B_{v+1} := A_v \cup (B_v \cap \VVc_{k,\HHc_0 \fixed[0.2]{\text{ }} \cup \fixed[0.2]{\text{ }} \{h_0\}})$. Then, it is clear that $A_v \subseteq A_{v+1} \subseteq B_{v+1} \subseteq B_v$ and $B_{v+1} \setminus A_{v+1} \subseteq k\cdot \HHb+h_0$, which, putting it all together, is enough (by induction) to conclude, since $X \subseteq A_1$, $B_1 \subseteq Y$ and $k \ge v+1$.
\end{proof}
Now, let $(A_n)_{n \ge 1}$ and $(B_n)_{n \ge 1}$ be as in the above claim, and set $A := \bigcup_{n\ge 1}A_n$. For a fixed $n \in \NNb^+$, both $A\setminus A_n$ and $B_n\setminus A$ are then subsets of $B_n \setminus A_n$, which, by condition \ref{item:recursive_(iii)}, is in turn contained in $k \cdot \HHb + h$ for some $h \in \NNb$ and $k \ge n$.

Accordingly, we get from here, Lemma \ref{lm:relaxed_(F2)_with_upper_arithmetic_bound}, and axiom \ref{item:F4b} that $\mu^\star(A\setminus A_n)\le \frac{1}{n}$ and $\mu^\star(B_n\setminus A) \le \frac{1}{n}$ for all $n$, with the result that
$$
\mu^\star(A)\le \mu^\star(A_n)+\mu^\star(A\setminus A_n) \le \mu^\star(A_n)+\frac{1}{n}<\xi+\frac{1}{n}
$$
and
$$
\mu^\star(A) \ge \mu^\star(B_n)-\mu^\star(B_n\setminus A) \ge \mu^\star(B_n)-\frac{1}{n} \ge \xi-\frac{1}{n},
$$
where, along with the subadditivity of $\mu^\ast$, we have used that $\mu^\ast(A_n) < \xi \le \mu^\ast(B_n)$ by condition \ref{item:recursive_(ii)}. Hence, $|\mu^\star(A)-\xi|\le \frac{1}{n}$ for every $n \in \NNb^+$, which is possible if and only if $\mu^\ast(A) = \xi$, and completes the proof when considering that $X \subseteq A \subseteq Y$.
\end{proof}

\section{Closing remarks}
\label{sec:final_remarks}
%
First, we show that the hypotheses of Theorem \ref{th:upperdarboux} are sharp, in the sense that, if we try to extend the theorem from upper quasi-densities to a larger class $\mathcal{C}$ of functions $\mathcal P(\HHb) \to \RRb$ by dropping either of conditions \ref{item:F3}-\ref{item:F5}, then the result breaks down in a somewhat dramatic way, since there exists a function $f \in \mathcal{C}$ such that its image is ``as far as possible'' from being the entire interval $[0,1]$, even under the additional assumptions that $f$ satisfies axiom \ref{item:F2} and $f(X) \le 1$ for each $X\subseteq \HHb$ (in particular, $f$ cannot have the weak Darboux property). This is the object of the following three examples, where for $X \subseteq \HHb$ we set $X^+ := X \cap \NNb^+$.
\begin{example}
\label{exa:(i)_without_(F3)}
Let $f$ be the function $\PPc(\HHb) \to \RRb$ defined as follows: Given $X \subseteq \HHb$, we assume $f(X) := 0$ if $|X^+| < \infty$, otherwise $f(X) := \sup_{n \ge 1} (x_{n+1} - x_n)^{-1}$,
where $(x_n)_{n \ge 1}$ is the natural enumeration of $X^+$. Of course, $f(\emptyset) = 0$, and it is straightforward that $f$ satisfies conditions \ref{item:F1}, \ref{item:F2} and \ref{item:F4b}, and $\imag(f) = \{0\} \cup \{\nicefrac{1}{k}: k \in \NNb^+\}$.
\end{example}
\begin{example}
\label{exa:(i)_without_(F4)}
Let $f$ be a monotone and subadditive function $\PPc(\HHb) \to \RRb$ such that $f(\HHb) = 1$ and $f(\emptyset) = 0$ (e.g., $f$ may be an upper density on $\HHb$), and let $f$ be the function
$\PPc(\HHb) \to \RRb$ mapping a set $X \subseteq \HHb$ to $1$ if $f(X) > 0$, and to $0$ otherwise, cf. \cite[Example 1]{LT}.

It is seen that $f$ satisfies axioms \ref{item:F1}-\ref{item:F3} and \ref{item:F5}, where we use, in particular, that $f(X \cup Y) \le f(X) + f(Y)$ for all $X, Y \subseteq \HHb$, and hence $f(X \cup Y) = 0$ whenever $f(X) = f(Y) = 0$. On the other hand, it is evident that $f(\emptyset) = 0$ and the image of $f$ is the $2$-element set $\{0, 1\}$.
\end{example}
\begin{example}
\label{exa:(i)_without_(F5)}
Let $f$ be the function $f: \PPc(\HHb) \to \RRb: X \mapsto (\inf(X^+))^{-1}$,
where $\inf(\emptyset) := \infty$ and $\frac{1}{\infty} := 0$. It is immediate that $f$ satisfies conditions \ref{item:F1}-\ref{item:F4}, cf. \cite[Example 3]{LT}. In addition, $f(\emptyset) = 0$ and the image of $f$ is the set $\{0\} \cup \{\nicefrac{1}{k}: k \in \mathbf N^+\}$.
\end{example}
Secondly, we may notice a kind of asymmetry in the definition of the strong Darboux property. More precisely, we say that $f$ has the \textit{symmetric \textup{(}strong\textup{)} Darboux property} if, for all $X, Y \subseteq \HHb$ with $X \subseteq Y$ and every $a \in [0,1]$, there exists a set $A$ with $X \subseteq A \subseteq Y$ and $f(A) = af(Y) + (1-a) f(X)$; this is the same as the strong Darboux property when $f$ is monotone, but is more general otherwise, as it does not require any longer that $f(X) \le f(Y)$.

Thus, we question if Theorem \ref{th:upperdarboux} can be extended to prove that upper quasi-densities satisfy the symmetric Darboux property. Yet, that is not the case, as shown by the following example.

\begin{example}
Let $\theta^\ast$ be the function
$$
\theta^\ast: \PPc(\HHb) \to \RRb: X \mapsto \left\{
\begin{array}{ll}
\!\! \dd^\ast(X) & \text{if } \iota(X) \le 3 \\
\!\! \frac{3}{4}(\iota(X)-1)\fixed[0.1]{\text{ }}\dd^\ast(X) & \text{if } 4 \le \iota(X) < \infty \\
\!\! 0 & \text{otherwise}
\end{array}
\!\!\right.,
$$
where $\dd^\ast$ is the upper asymptotic density on $\ZZb$ and $\iota$ the function $\PPc(\HHb) \to \NNb^+ \cup \{\infty\}$ taking a set $X \subseteq \HHb$ to the infimum of the integers $n \ge 1$ for which there is $Y \subseteq \HHb$ such that
$\dd^\ast(Y) \ge \frac{1}{n}$ and $|(q \cdot Y + r) \setminus X| < \infty$ for some $q \in \NNb^+$ and $r \in \ZZb$,
with the convention that $\inf(\emptyset) := \infty$.

We have from \cite[Lemmas 2 and 3]{LT} that $\theta^\ast$ is a non-monotone upper quasi-density, so we just need to prove that $\theta^\ast$ does not have the symmetric Darboux property.

For this, set $X := \bigcup_{n \ge 2} \!\left\llb \frac{1}{4} (2n-1)! + \frac{3}{4} (2n)! + 1, (2n)! + 1 \right\rrb$ and $Y := X \cup (4\cdot \HHb)$; note that we are intentionally starting with $n = 2$ in the definition of $X$, so as to discard the integers $\le 3$
and have that $Y$ can be written as the union of $4 \cdot \HHb$ and $\bigcup_{h=1}^3 (X \cap (4 \cdot \HHb + h))$, which otherwise would be false for $\HHb = \NNb^+$. Then, using that $\dd^\ast$ satisfies \ref{item:F3} and \ref{item:F4b} yields, together with \cite[Lemma 1]{LT}, that $\dd^\ast(X) = \frac{1}{4}$ and
$$
\dd^\star(Y) \le \dd^\star(4\cdot \HHb)+\sum_{h=1}^3 \dd^\star(X \cup (4\cdot \HHb+h))=\frac{1}{4}+\frac{3}{16} = \frac{7}{16}.
$$

On the other hand, it follows from the claim established within the proof of \cite[Lemma 3]{LT} and the observation made on the first line of the same that $\iota(X) = 4$ and $\iota(Y) = 1$.

Now, consider a set $A$ such that $X \subseteq A \subseteq Y$. By \cite[Lemma 3]{LT}, we have $\iota(A) \le \iota(X) = 4$, and so there are two cases: Either $\iota(A) \le 3$, and then $\theta^\star(A)=\dd^\star(A) \le \dd^\star(Y) \le \frac{7}{16} < \frac{1}{2}$, or $\iota(A) = 4$, and then
$\theta^\star(A) = \frac{9}{4}\fixed[0.1]{\text{ }}\dd^\ast(A) \ge \frac{9}{4} \dd^\star(X)= \frac{9}{16} > \frac{1}{2}$.
This shows that $\theta^\ast(A)$ cannot attain the value $\frac{1}{2}$, so $\theta^\ast$ does not have the symmetric Darboux property.
\end{example}

%
Lastly, suppose $\mu^\ast$ is an upper [quasi-]density on $\HHb$ with associated lower [quasi-]density $\mu_\ast$, and denote by $\mu$ the partial function $\PPc(\HHb) \pto \RRb$ obtained by restriction of $\mu^\ast$ to the set $\{X\subseteq \HHb: \mu^\ast(X)=\mu_\ast(X)\}$. We refer to $\mu$ as the \textit{\textup{[}quasi-\textup{]}density} induced by $\mu^\ast$.

In the light of our main result, it is natural to ask whether $\mu$ must have the strong Darboux property: It occurs that the answer is affirmative if $\mu^\star$ is additive, as in that case $\dom(\mu)$ is just $\PPc(\HHb)$, and $\mu$ has the strong Darboux property by Theorem \ref{th:upperdarboux}. (The existence of additive upper densities is provable in ZFC, but independent of ZF, see \cite[Remark 3]{LT}.) However, we ignore what happens in general, so we raise the following question, which has a positive answer in the case when $\mu^\ast$ is either the upper Buck density or the upper Banach density, see \cite[Theorem 2.1]{PaSa} and \cite[Theorem 4.2]{GLS} respectively.
\begin{question*}
\label{quest:strong_darboux_of_induced_densities}
Does a [quasi-]density $\mu$ need to have the strong Darboux property? If not, what about the weak Darboux property?
\end{question*}
A few more questions along these lines can be found in \cite[\S{} 8]{LT}, see in particular Question 3 therein, and involve a kind of ``joint weak Darboux property'' in the spirit of \cite[Theorem 2]{GMT}, where the focus is, however, on (upper and lower) weighted densities.
\section*{Acknowledgments}
P.L. was supported by a PhD scholarship from Universit\`a ``Luigi Bocconi'', and S.T. by NPRP grant No. [5-101-1-025] from the Qatar National Research Fund (a member of Qatar Foundation) and partly by the French ANR Project ANR-12-BS01-0011.
The statements made herein are solely the responsibility of the authors.

The authors are grateful to Ladislav Mi\v{s}\'ik (University of Ostrava, SK) for having suggested Example \ref{exa:(i)_without_(F3)}, to Carlo Sanna (Universit\`a di Torino, IT) for many useful comments, and to Nordine Mir (Texas A\&M University at Qatar, Q) for fruitful discussions.

\end{document}